\documentclass[12pt,twoside]{amsart}
\usepackage{amsmath}
\usepackage{color}
\usepackage{amssymb}
\usepackage{wasysym}
\usepackage{psfrag}
\usepackage{graphicx,epsf,amsmath}  % standard LaTeX graphi       cs tool
\usepackage{epsf,graphicx}
\usepackage{comment}
   
\newtheorem{thm}{Theorem}[section]
\newtheorem{theorem}{Theorem}[section]

\newtheorem{remark}[thm]{Remark}

\newtheorem{lemma}[thm]{Lemma}

\newtheorem{corollary}[thm]{Corollary}

  \def\R{{\mathbb R}}
  \def\RR{{\mathbb R}}

\newcommand{\eps}{\varepsilon}

\title[fractional harmonic flow]{Weak solutions of geometric flows associated to integro-differential harmonic maps}
\author[A. Schikorra]{Armin Schikorra}
\address{University of Basel}
\email{armin.schikorra@unibas.ch}
\author[Y. Sire]{Yannick Sire} 
\address{Johns Hopkins University, Krieger Hall, Baltimore, USA}
\email{sire@math.jhu.edu}
\author[C. Wang]{Changyou Wang} 
\address{Purdue University, USA}
\email{wang2482@purdue.edu}

\begin{document}
\begin{abstract}
The purpose of this note is to prove the existence of global weak solutions to the flow associated to integro-differential harmonic maps into spheres and Riemannian homogeneous manifolds.   

\end{abstract}
\maketitle
\tableofcontents
\section{Introduction}

Let $\mathcal N$ be a compact smooth connected manifold without boundary, isometrically embedded into $\mathbb R^L$,
and $\Omega$ be a smooth bounded domain of  $\RR^n$ with $n \geq 1$. Consider the Dirichlet energy functional
$$
E_2(u)=\int_\Omega |\nabla u|^2\,dx,
$$ 
where $u$ maps $\Omega$ into $\mathcal N$. The harmonic map flow is then the flow defined by: given $u_0$ an initial data in a suitable functional space, consider  
\begin{equation}
 \begin{cases}
  \partial_t u +  E_2'(u,\cdot) \perp T_u \mathcal{N} \quad &\mbox{in $\Omega \times (0,\infty)$,}\\
   u(\cdot,0) = u_0 \quad &\mbox{in $\Omega$,}\\	
   u(\cdot,t) = u_1 \quad &\mbox{on $\partial \Omega \times (0,\infty)$.}
 \end{cases}
\end{equation}
In  the previous expression, $E_2'$ denotes the Frechet derivative of the functional $E_2$. This flow has been studied in many papers (see e.g. \cite{chen,chenStruwe,struweJDG} and the monograph \cite{bookLW}). 

In the present paper, we investigate the flow associated to the critical points of the following type of nonlocal energy:
\begin{equation}\label{energy}
E_{s,p}(u)=\frac1p\int_{\Omega\times \Omega} \frac{|u(x)-u(y)|^p}{|x-y|^{n+sp}}\,dx\,dy \ \big(\equiv \big[u\big]_{\dot{W}^{s,p}(\Omega)}^p\big),
\end{equation}
where $1<p<+\infty$, $s\in (0,1)$,  $\Omega$ is a Lipschitz domain in $\RR^n$ (possibly the whole $\RR^n$), $n \geq 1$,  
and $u$ takes values into $\mathcal N\subset\mathbb R^L$. We define the Sobolev space
$$W^{s,p}(\Omega,\mathcal N)=\Big\{u\in W^{s,p}(\Omega, \mathbb R^L): \ u(x)\in \mathcal N
\ {\rm{for\ a.e.\ }}x\in\Omega\Big\}.$$
The quantity $E_{s,p}(u)$ is known as the Gagliardo norm of $u$ and has been investigated in the context of harmonic maps (i.e. with the side-condition that $u$ maps into $\mathcal{N}$) by several authors. 
In the case, $n=1$, $s=\frac12$ and $p=2$, the regularity of such maps has been obtained in \cite{LR1,LR2}. In the case of general $n\geq 2$, $s = \frac{n}{2}$ it has been investigated in \cite{ArminSphereHarmMaps,DLmanifolds}. In the case, $p=\frac{n}{s}$ for any $n\geq 1$ and $s \in (0,1)$, this has been investigated by one of the authors \cite{armin} in the case of $\mathcal N=\mathbb S^{L-1}$. In all the previous results, the exponent $p$ is the conformal exponent $p = \frac{n}{s}$, and in this case one expects regularity to hold everywhere.  
% under consideration are invariant under conformal transformations in $\RR^N$, 
% hence giving full regularity of critical points. 
If $p < \frac{n}{s}$ one expects only partial regularity and no results are available in the literature in this case.  
In the present paper, we introduce and investigate the flow associated to these maps in all dimensions and for any $p$ and $s$ in the above ranges.

\begin{remark}
In the paper \cite{millotSire}, a Ginzburg-Landau approximation of $\frac12$-harmonic maps is performed. Furthermore, it appears that $\frac12$-harmonic maps can be formulated as harmonic maps with free boundaries {\rm{(}}see \cite{hardtLin}{\rm{)}}. This approach is a new feature of fractional harmonic maps, that we will not exploit here. 
\end{remark} 

We now recall the flow associated with $E_{s,p}(\cdot)$ for maps from $\Omega$ to $\mathcal N$: given $u_0:\Omega\to\mathcal N$ an initial data in a suitable functional space, consider  $u:\Omega\times [0,\infty)\to\mathcal N$ satisfying
\begin{equation}\label{eq:harmonicmapflow}
 \begin{cases}
  \partial_t u +  E_{s,p}'(u,\cdot) \perp T_u \mathcal{N} \quad &\mbox{in $\Omega \times (0,\infty)$,}\\
   u(\cdot,0) = u_0 \quad &\mbox{in $\Omega$,}\\	
   u(\cdot,t) = u_1 \quad &\mbox{on $\partial \Omega \times (0,\infty)$.}
 \end{cases}
\end{equation}

Here, $E_{s,p}'$ denotes the Frechet derivative of $E_{s,p}$ in the function space $W^{s,p}(\Omega,\R^L)$: for $v\in W^{s,p}(\Omega, \mathcal N)$ and $\varphi\in W^{s,p}_0(\Omega, \mathbb{R}^L)$,
\begin{equation}\label{nonlocaloperator}
E_{s,p}'(v, \varphi)
=\int_{\Omega^2}\frac{|v(x)-v(y)|^{p-2}(v(x)-v(y))(\varphi(x)-\varphi(y))}{|x-y|^{n+sp}}.
\end{equation}
Our main result is the following. It provides the existence of global {\sl weak} solutions to the flow (\ref{eq:harmonicmapflow}).
\begin{theorem}\label{main}
Let $\mathcal{N} = \mathbb{S}^{L-1} \subset \RR^L$ be the unit sphere or $\mathcal N$ be a Riemannian homogeneous manifold that is equivariantly  embedded into   $\mathbb R^L$. For any $u_0\in W^{s,p}(\Omega,\mathcal N)$, there exists $u: \Omega \times (0,\infty) \to \mathcal N$ such that
\begin{equation}\label{eqflow}
 \begin{cases}
  \partial_t u + E_{s,p}'(u,\cdot) \perp T_u \mathcal N \quad &\mbox{in $\Omega \times (0,\infty)$},\\
  u(\cdot,0) = u_0 \quad &\mbox{in $\Omega$},\\
 \end{cases}
\end{equation}
in the following sense: 
\begin{itemize}
\item $u \in L^\infty(0, \infty; W^{s,p}(\Omega))$.
\item $\partial_t u \in L^2(0,\infty; L^2(\Omega))$.
%\item $u(t,x) \in \mathcal N$ a.e. in $\Omega$
\item $u$ satisfies the first equation in \eqref{eqflow} in $\mathcal D'(\Omega\times (0,\infty))$\footnote{See
(\ref{eq:flow3}) in Section 5.}, and
$$\lim_{t\rightarrow 0^+}\big\|u(\cdot, t)-u_0\big\|_{L^2(\Omega)}=0.$$ 
\end{itemize}
\end{theorem}

The attentive reader might have noticed that in \eqref{eqflow} the third equation of \eqref{eq:harmonicmapflow} is missing. In our construction one can easily prescribe the boundary $u_1$ on $\partial \Omega$ in a distributional sense. However, if $s < 1/p$ this is meaningless, since the trace of $W^{s,p}(\Omega)$-functions belongs to $W^{s-1/p,p}(\partial\Omega)$. One way to avoid this issue is to replace the third equation in \eqref{eq:harmonicmapflow} by
\[
    u(\cdot,t) = u_1 \quad \mbox{on $(\RR^n\backslash \Omega) \times (0,\infty)$,}
\]
and then find a flow $u: (0,+\infty) \to W^{s,p}(\RR^n,\mathcal N)$. This is an easy adaptation of the arguments we present below, and we leave it to the interested reader.

Let us remark that in \cite{Pu-Guo2011} they studied the fractional Landau-Lifshitz-Gilbert equation and in particular obtain Theorem~\ref{main} in the case of the $2$-sphere. Also, we would like to point out that the existence of global weak solutions to the $p$-harmonic maps to $\mathbb S^{L-1}$
has been proven by \cite{chen} and \cite{ChenHongHungerbu}.

\section{Preliminaries}
We consider $\Omega \subset \RR^n$ to be bounded. A simple adaption of the argument provides the case  $\Omega = \RR^n$. 

\begin{lemma}\label{la:compactness}
Let $\Omega \subset \RR^n$ be a bounded domain and $K \subset \RR^L$ a compact set. For any $s > 0$, $p \in (1,\infty)$, and any $q \in [1,\infty)$,
$W^{s,p}(\Omega,K)$ is embedded compactly into $L^q(\Omega,K)$.
\end{lemma}
\begin{proof}
By interpolation we know that $W^{s,p}(\Omega,\RR^L)$ is embedded compactly into $L^1(\Omega,\RR^L)$. On the other hand, $W^{s,p}(\Omega, K) \subset L^\infty(\Omega)$ by the boundedness of $K$. In particular,  $W^{s,p}(\Omega,K)$ is embedded compactly into $L^q(\Omega,K)$ for all $q \in [1,\infty)$.
\end{proof}

For the classical $p$-Laplacian,  it is \emph{not} true that $f_k \rightharpoonup f$ weakly in $W^{1,p}(\Omega)$ implies
that $|\nabla f_k|^{p-2}\nabla f_k\rightharpoonup |\nabla f|^{p-2}\nabla f$ weakly in $L^{\frac{p}{p-1}}(\Omega)$, i.e.,
\[
 \lim_{k \to \infty} \int_\Omega |\nabla f_k|^{p-2} \nabla f_k\ \nabla \varphi 
 = \int_\Omega |\nabla f|^{p-2} \nabla f\ \nabla \varphi, \  \varphi \in W^{1,p}(\Omega).
 \]
A nice feature of the \emph{fractional} $p$-Laplacian is that an analogue is actually \emph{true}. In fact, denote 
$$\Omega^2=\Omega\times \Omega, \ \Omega_T=\Omega\times [0,T],
\ {\rm{and}} \ \Omega^2_T=\Omega^2\times [0,T].$$
For $0<T\le\infty$, we have 
the following property.
\begin{lemma}\label{la:gagliardounderweakconv} For any $0<T<\infty$, 
assume that $f_k\rightarrow f$ a.e. in $\Omega_T$,
$f_k \rightharpoonup f$  in $L^p(0,T;W^{s,p}(\Omega))$. Also assume that $\|g_k\|_{L^\infty(\Omega_T)}$ is bounded
and  $g_k\rightarrow g$ a.e. in $\Omega_T$. 
Then for any $\varphi \in L^p(0,T; W^{s,p}(\Omega))$ it holds that
\[
\begin{split}
 &\int_{\Omega^2_T} 
 \frac{|f_k(x,t)-f_k(y,t)|^{p-2}(f_k(x,t)-f_k(y,t))g_k(y,t)(\varphi(x,t)-\varphi(y,t))}{|x-y|^{n+sp}}\\
 &\xrightarrow{k\rightarrow\infty}\\
 &\int_{\Omega^2_T} \frac{|f(x,t)-f(y,t)|^{p-2}(f(x,t)-f(y,t))g(y,t)(\varphi(x,t)-\varphi(y,t))}{|x-y|^{n+sp}}. 
 \end{split}
\]
\end{lemma}
\begin{proof} From the assumptions on $f_k$ and $g_k$, we have that
\[\begin{split}
&c_k(x,y,t):=\frac{|f_k(x,t)-f_k(y,t)|^{p-2}(f_k(x,t)-f_k(y,t))g_k(y,t)}{|x-y|^{(\frac{n}{p}+s)(p-1)}}\\
&\xrightarrow{k\rightarrow\infty}
c(x,y,t):=\frac{|f(x,t)-f(y,t)|^{p-2}(f(x,t)-f(y,t))g(y,t)}{|x-y|^{(\frac{n}{p}+s)(p-1)}},
\end{split}
\]
for a.e. $(x,y,t)\in\Omega^2_T$. Furthermore, direct calculations imply that
$$
\int_{\Omega^2_T} |c_k(x,y,t)|^{\frac{p}{p-1}}
\le C\big\|g_k\big\|_{L^\infty(\Omega_T)}^{\frac{p}{p-1}} \big\|f_k\big\|_{L^p(0,T; W^{s,p}(\Omega))}^p
\le C.
$$
We conclude that
$c_k(x,y,t)\rightharpoonup c(x,y,t)$ in $L^{\frac{p}{p-1}}(\Omega^2_T)$ as $k\rightarrow\infty$. 
On other hand, if we set
$$\Phi(x,y,t):=\frac{\varphi(x,t)-\varphi(y,t)}{|x-y|^{\frac{n}{p}+s}}, \ (x,y,t)\in\Omega^2_T, $$
then $\Phi\in L^p(\Omega^2_T)$ and
$$\big\|\Phi\big\|_{L^p(\Omega^2_T)}\le C\, \big\|\varphi\big\|_{L^p(0,T; W^{s,p}(\Omega))}^p.$$
Thus we  obtain
$$\int_{\Omega^2_T} c_k(x,y,t)\Phi(x,y,t)\xrightarrow{k\rightarrow\infty}\int_{\Omega^2_T} c(x,y,t) \Phi(x,y,t).$$
This completes the proof. 
\end{proof}

As a direct consequence of Lemma \ref{la:gagliardounderweakconv}, we have
\begin{corollary}\label{la:gagliardounderweakconv1}
Let $f_k \rightharpoonup f$  in $W^{s,p}(\Omega)$, then for any $\varphi \in W^{s,p}(\Omega)$
\[
\begin{split}
 &\lim_{k \to \infty}\int_{\Omega^2} \frac{|f_k(x)-f_k(y)|^{p-2}(f_k(x)-f_k(y))(\varphi(x)-\varphi(y))}{|x-y|^{n+sp}}\\
 &= \int_{\Omega^2} \frac{|f(x)-f(y)|^{p-2}(f(x)-f(y))(\varphi(x)-\varphi(y))}{|x-y|^{n+sp}}. 
 \end{split}
\]
\end{corollary}

We also need the following fact in the proof of Theorem 1.1.
\begin{lemma}
The functional $E_{s,p}$ is sequentially lower semi-continuous with respect to the weak topology on $W^{s,p}(\Omega,\RR^L)$.
\end{lemma}
\begin{proof} Assume that $f_k \rightharpoonup f$ in $W^{s,p}(\Omega,\RR^L)$. Now we use a duality argument: It holds that
\[
\begin{split}
 &\left (\int_{\Omega^2} \frac{|f(x)-f(y)|^p}{|x-y|^{n+sp}} \right )^{\frac{p}{p-1}}\\
 &\leq \sup_{\varphi} \int_{\Omega^2} \frac{|f(x)-f(y)|^{p-2}(f(x)-f(y))(\varphi(x)-\varphi(y))}{|x-y|^{n+sp}},
 \end{split}
\]
where the supremum is taken over $\varphi \in W^{s,p}({\Omega},\RR^L)$, with $\|\varphi\|_{W^{s,p}(\Omega)} \leq 1$. In particular, for any $\eps > 0$ we find some $\varphi_\eps \in W^{s,p}({\Omega}, \RR^L)$, with $\|\varphi_\eps\|_{W^{s,p}(\Omega)} \leq 1$, such that
\[
\begin{split}
 &\left (\int_{\Omega^2} \frac{|f(x)-f(y)|^p}{|x-y|^{n+sp}}\ dx dy \right )^{\frac{p}{p-1}}\\
 \leq &(1+\eps)\int_{\Omega^2} \frac{|f(x)-f(y)|^{p-2}(f(x)-f(y))(\varphi_\eps(x)-\varphi_\eps(y))}{|x-y|^{n+sp}}.
 \end{split}
\]
By Lemma~\ref{la:gagliardounderweakconv}, the right-hand side of the above inequality is equal to
\[
 (1+\eps)\lim_{k\to \infty}\int_{\Omega^2} \frac{|f_k(x)-f_k(y)|^{p-2}(f_k(x)-f_k(y))(\varphi_\eps(x)-\varphi_\eps(y))}{|x-y|^{n+sp}},
\]
which, by H\"older's inequality, can be bounded by
\[
\leq  (1+\eps)\liminf_{k\to \infty} \left (\int_{\Omega^2} \frac{|f_k(x)-f_k(y)|^p}{|x-y|^{n+sp}}\ dx dy \right )^{\frac{p}{p-1}}.
\]
Sending $\eps \to 0$, we then obtain the conclusion.
\end{proof}

% 
% 
% \begin{defi}\label{def:admissible}
% A map $E: L^p(\Omega,\RR^M) \to  [0,\infty]$ shall be called an \emph{admissible energy functional}, if
% \begin{enumerate}
%  \item[(i)] there exists a norm $\|\cdot\|_{\mathbb{E}}$ allowed to be $+\infty$ on $L^p(\Omega,\RR^M)$ comparable to
%  \[
%   \|f\|_{\mathbb{E}} \approx \|f\|_{L^p(\Omega)} + (E(f))^{\frac{1}{p}}.
%  \]
%  Here we define the reflexive space $(\mathbb{E},\|\cdot\|_{\mathbb{E}})$ as the space of all maps $f \in L^p(\Omega,\RR^M)$ such that $\|f\|_{\mathbb{E}} < \infty$.
% \item[(ii)]  for any compact $K \subset \RR^M$, the embedding $\mathbb{E} \cap L^p(\Omega,K) \subset L^p(\RR^M,K)$ is compact. 
% \item[(iii)] . 
% \item[(iv)]
% For any $f \in \mathbb{E}$ and any $\varphi \in C_0^\infty(\Omega,\RR^M)$
% \[
%  E'(f,\varphi) = \frac{d}{dt} \big |_{t = 0} E(f+t\varphi)
% \]
% is well defined. 
% \end{enumerate}
% \end{defi}

\section{Proof of Theorem~\ref{main}: Step 1}\label{sec:step1}
Step 1 is the time discretization of the flow (\ref{eq:harmonicmapflow}). This step works for any compact Riemannian manifold $\mathcal N\subset\mathbb R^L$ without boundary. To the best of our knowledge this idea is due to an unpublished work by Kukuchi.
Fix $h > 0$. Starting from $u_0\in W^{s,p}(\Omega, \mathcal N)$, for $k\ge 1$ we define $u^k \in W^{s,p}(\Omega,\mathcal N)$ to be a minimizer of 
${\mathcal E}_k$ in the space of maps $v \in W^{s,p}(\Omega,\mathcal N)$, where 
\[
 {\mathcal E}_k(v) := E_{s,p}(v) + \int_{\Omega} \frac{|v-u^{k-1}|^2}{2h}.
\]
This existence of $u^k$ follows from the direct method of calculus of variations, since ${\mathcal E}_k(\cdot)$ is coercive and sequentially lower semicontinuous 
with respect to the weak topology in $W^{s,p}(\Omega)$. Also observe the compactness of $\mathcal{N}$ guarantees that 
$u^k(x) \in \mathcal N$ for a.e. $x \in \Omega$, see Lemma~\ref{la:compactness}.

By direct calculations, $u^k$ satisfies the Euler-Lagrange equation:
\begin{eqnarray}\label{discreteEL}
&&\int_\Omega \frac{u^k-u^{k-1}}{h} \varphi^k\nonumber\\
&&=-\int_{\Omega^2}\frac{|u^k(x)-u^k(y)|^{p-2}(u^k(x)-u^k(y))(\varphi^k(x)-\varphi^k(y))}{|x-y|^{n+sp}},
\end{eqnarray}
where $\varphi^k=P_{\mathcal N}(u^k)(\psi)$,  $P_{\mathcal N}(y):\mathbb R^L\to T_y\mathcal N$, for $y\in\mathcal N$,
is the orthogonal projection\footnote{It is well-known that there exists small $\delta>0$ depending
only on $\mathcal N$ such that there exists a smooth nearest point projection map $\Pi_{\mathcal N}$ from $\mathcal N_\delta$, the $\delta$-neighborhood of $\mathcal N$,
to $\mathcal N$, see also \cite[Appendix 2.12.3]{SimonETH}. Note that $P_{\mathcal N}(y)=\nabla\Pi_{\mathcal N}(y)$ for $y\in\mathcal N$. We revisit this fact in Lemma~\ref{la:Projection}.},
and $\psi\in C^\infty_0(\Omega,\mathbb R^L)$. 

It is clear, by the minimality and comparison, that for $j=1,\cdots, k$,
\[
 \mathcal E_j(u^j)=E_{s,p}(u^j)+\int_\Omega \frac{|u^j-u^{j-1}|^2}{2h}\leq \mathcal E_{j}(u^{j-1}) = E_{s,p}(u^{j-1}).
\]
Taking summation over $j=1,\cdots, k$, we obtain that 
\begin{equation}\label{discreteenergyinequality}
 E_{s,p}(u^k) +\sum_{j=1}^k \int_\Omega \frac{|u^j-u^{j-1}|^2}{2h} \leq E_{s,p}(u^0), \ \forall\ k\ge 1.
\end{equation}
Here we denote $u^0=u_0$.
Now we define $u^h:\Omega\times [0,\infty)\to\mathcal N$ and $v^h:  \Omega\times [0,\infty)\to\mathbb R^L$ by letting
\[\displaystyle
 \begin{cases} u^h(x,t) = u^k(x), \\ v^h(x,t)=\frac{t-kh}{h} u^{k+1}(x)+\frac{(k+1)h-t}{h} u^k(x),\end{cases}
 \]
when $x\in\Omega,\ t \in [kh,(k+1)h)$ for some $k\ge 0$.

It follows from (\ref{discreteenergyinequality}) that for any $h>0$, the following energy inequality holds:
\begin{equation}\label{continousenergyinequality1}
 E_{s,p}(u^h(\cdot,t)) +\frac12\int_0^t\int_\Omega |\partial_t v^h|^2 \leq E_{s,p}(u_0), \ \forall\ t>0.
\end{equation}
Since $\mathcal{N}$ is compact, it is easy to see that there exists $C>0$ depending only on $\mathcal N$ such that for any $h>0$, 
\begin{equation}\label{pointbound}
|v^h(x,t)|\le C, \ \forall\ x\in\Omega, t>0.
\end{equation}
Moreover, for any $h>0$ and $t>0$, we have, by convexity and (\ref{discreteenergyinequality}),
\begin{eqnarray}\label{energybound}
E_{s,p}(v^h(\cdot, t))&\le& \frac{t-kh}{h} E_{s,p}(u^{k+1})+\frac{(k+1)h-t}{h} E_{s,p}(u^{k})\nonumber\\
&\le& E_{s,p}(u_0),
\end{eqnarray}
where $k\ge 0$ is chosen so that $kh\le t<(k+1)h$.

It follows from (\ref{continousenergyinequality1}), (\ref{pointbound}), and (\ref{energybound})
that there exist $u, v: \Omega\times (0,\infty)\to\mathbb R^L$ such that after passing to a subsequence,
$$v^h\rightarrow v \ {\rm{in}}\ L^2_{\rm{loc}}(\Omega\times (0,\infty)),\ (\partial_t v^h, \nabla v^h)\rightharpoonup (\partial_t v, \nabla v)  \ {\rm{in}}\ L^2_{\rm{loc}}(\Omega\times (0,\infty)),$$
and
$$
u^h\rightharpoonup u \ {\rm{in}}\ L^p_{\rm{loc}}((0,\infty),W^{s,p}(\Omega))
% , \
% \nabla u^h\rightharpoonup \nabla u  \ {\rm{in}}\ L^2_{\rm{loc}}(\Omega\times (0,\infty))
,
$$
as $h\rightarrow 0$.

Note that from the definitions of $u^h$ and $v^h$, and (\ref{discreteenergyinequality}) that there exists $C>0$ such that for any $0<T<+\infty$
\begin{equation}\label{l2closedness}
\int_0^T\int_\Omega |u^h-v^h|^2\le Ch\sum_{i=1}^{[\frac{T}h]+1}\int_\Omega |u^{i+1}-u^i|^2 \le Ch^2 E_{s,p}(u_0),
\end{equation}
where $[\frac{T}{h}]$ denotes the largest integer part of $\frac{T}{h}$. As an immediate consequence of (\ref{l2closedness}), we
obtain that $u\equiv v$ in $\Omega\times (0,\infty)$. Moreover, since $u^h(x,t)\in \mathcal N$ for a.e. $(x,t)\in\Omega\times (0,\infty)$,
(\ref{l2closedness}) also implies that $u(x,t)\in \mathcal N$ for a.e. $(x,t)\in\Omega\times (0,\infty)$. By the sequential lower semicontinuity of $E_{s,p}$ and
(\ref{continousenergyinequality1}), we conclude that  for a.e. $t>0$, it holds
\begin{equation}\label{continousenergyinequality2}
 E_{s,p}(u(\cdot,t)) +\frac12\int_0^t\int_\Omega |\partial_t u|^2 \leq E_{s,p}(u_0).
\end{equation}
In particular, $u\in L^\infty((0,\infty), W^{s,p}(\Omega,\mathcal N))$ and $\partial_t u\in L^2(\Omega\times (0,\infty))$.
%\[
% \sup_{h > 0} \sup_{t >0} E_{s,p}(u^h(\cdot,t)) \leq E_{s,p}(u^0).
%\]
%Moreover, for any fixed $t \in [0,\infty)$,
%\[
%\begin{split}
% &\| u^h(\cdot,t)\|_{L^2(\Omega)}^2 + \| u^h(\cdot,t)\|_{\dot{W}^{s,p}(\Omega)}^p \leq C\big[h \mathcal E_k(u^h) +  C\| u^h(\cdot,t-h)\|_{L^2}^2\big] \\
% &\leq C h E_{s,p}(u^0) \lceil \frac{t}{h} \rceil \leq C(t+1).\\
% \end{split}
%\]
%By compactness, for any fixed $t$ letting some sequence $h \to 0$ we obtain a limit function $u(\cdot,t) \in \mathbb{E}$, $u \in \mathbb{S}$ a.e.. It holds,
%\[
% E(u(\cdot,s)) \leq E(u(\cdot,t)) \quad \mbox{for any $s > t$}.
%\]
%In particular $u \in L^\infty((0,\infty),\mathbb{E}))$.

%On the other hand, for any $L \in \mathbb{N}$,
%\[
% \begin{split}
% &\int_{0}^{Lh} \int_{\Omega} \left (\frac{|u^h(x,t+h)-u^h(x,t)|}{h} \right )^2\ dx\ dt\\
% =&h^{-1} \sum_{k=1}^L \int_{(k-1)h}^{kh} \int_{\Omega} \frac{|u^{k}(x)-u^{k-1}(x)|^2}{h} \ dx\ dt\\
%=& \sum_{k=1}^L \int_{\Omega} \frac{|u^{k}(x)-u^{k-1}(x)|^2}{h} \ dx\\
%\leq& 2\sum_{k=1}^L (E_k(u^k) - E(u^k))\\
%\leq& 2\sum_{k=1}^L (E(u^{k-1}) - E(u^k))\\
%=&2 (E(u^{0}) - E(u^L)).
%\end{split}\]
%and in particular,
%\[
%  \int_{0}^{\infty} \int_{\Omega} \left (\frac{|u^h(x,t+h)-u^h(x,t)|}{h} \right )^2\ dx\ dt \leq 2\ E(u^0).
%\]
%Thus, letting $h \to 0$, we obtain that $\partial_t u \in L^2((0,\infty),L^2(\Omega))$, and
%\[
%  \int_{0}^{\infty} \int_{\Omega} \left (\partial_t u \right )^2\ dx\ dt \leq 2\ E(u^0).
%\]
Finally, it follows from the Euler-Lagrange equation (\ref{discreteEL}) that  
\begin{eqnarray}\label{discreteEL1}
&&\int_{\Omega^2_T}\frac{|u^h(x,t)-u^h(y,t)|^{p-2}(u^h(x,t)-u^h(y,t))(\varphi(x,t)-\varphi(y,t))}{|x-y|^{n+sp}}
\nonumber\\
&&\qquad\qquad\ \ \ \ \ \ \ \ \ \ \ \ \ \ \ \ \ \ \ \ \ \ \ \ \ \ \ \ \ \  +\int_{\Omega_T} \partial_t v^h(x,t) \varphi(x,t)=0,
\end{eqnarray}
for any $\varphi\in L^\infty(0,T; W^{s,p}_0(\Omega,T_{u^h(\cdot, t)}\mathcal N))\cap L^\infty(\Omega_T,\mathbb R^L).$
\footnote{\ Here we denote the Sobolev space
$$W^{s,p}_0(\Omega, T_{u^h(\cdot, t)}\mathbb S^{L-1}):=\Big\{\psi\in W^{s,p}_0(\Omega,\mathbb R^L): \ \psi(x,t)\in T_{u^h(x,t)}\mathbb S^{L-1}
\ {\rm{a.e.}}\ x\in\Omega\Big\}.$$}

\noindent{\it Step 2}: $u$ satisfies the equation (\ref{eqflow}) in the sense of distributions
when $\mathcal N=\mathbb S^{L-1}$ or when $\mathcal N$ is a compact homogenous Riemannian manifold. 
In the case of the sphere, this follows from the arguments in Section~\ref{sec:sphere}. For homogeneous Riemannian manifolds see Section~\ref{sec:homo}. 

\section{Proof of Theorem~\ref{main}: Step 2 - the case of a sphere}\label{sec:sphere}
Let $u$ be constructed as in Section~\ref{sec:step1}. Assume that $\mathcal{N} = \mathbb{S}^{L-1}$. Then $u$ satisfies the equation (\ref{eqflow}) in the sense of distributions
when $\mathcal N=\mathbb S^{L-1}$. That is, for any $0<T<\infty$,
\begin{equation}\label{eq:flow1}
\int_{\Omega_T} \langle\partial_t u(x,t),\varphi(x,t)\rangle+\int_0^TE_{s,p}'(u(\cdot, t), \varphi(\cdot, t))=0, 
\end{equation}
for all $\varphi\in L^\infty(0,T; W^{s,p}_0(\Omega, T_{u(\cdot, t)}\mathbb S^{L-1}))\cap L^\infty(\Omega_T, \mathbb R^L)$.
%Let $u_k,u \in W^{s,p}(\Omega,\mathbb S^{L-1})$ and $u_k\rightharpoonup u$ in $W^{s,p}(\Omega,\RR^L)$.
%Then for any $\varphi\in C_0^\infty(\Omega,\mathbb R^L)$ it holds that
%\begin{equation}\label{operatorconvergence}
% E_{s,p}'\big(u_k,P_{\mathbb S^{L-1}}(u_k)(\varphi)\big) 
% \xrightarrow{k \to \infty} E_{s,p}'\big(u,P_{\mathbb S^{L-1}}(u)(\varphi)\big).
%\end{equation}

For any $\psi\in L^\infty(0,T; W^{s,p}_0(\Omega, \mathbb R^L))\cap L^\infty(\Omega_T,\mathbb R^L)$, we can check that
$$\varphi^h=u^h\times \psi\in L^\infty(0,T; W^{s,p}_0(\Omega, T_{u^h(\cdot, t)}\mathbb S^{L-1}))
\cap L^\infty(\Omega_T,\mathbb R^L).$$ 
Substituting $\varphi^h$ into (\ref{discreteEL1}), we obtain that
$$\int_{\Omega_T} \langle\partial_t v^h(x,t), u^h(\cdot, t)\times \psi(\cdot, t)\rangle+
\int_0^T E_{s,p}'(u^h(\cdot, t), u^h(\cdot, t)\times\psi(\cdot, t))=0.$$
It is readily seen that
$$\int_{\Omega_T} \langle\partial_t v^h(x,t), u^h(\cdot, t)\times \psi(\cdot, t)\rangle\xrightarrow{h\to 0}
\int_{\Omega_T} \langle\partial_t u(x,t), u(\cdot, t)\times \psi(\cdot, t)\rangle.$$
Now we want to apply Lemma~\ref{la:gagliardounderweakconv} and the symmetry of $\mathbb S^{L-1}$ 
\footnote{See \cite{chen} for the heat flow of harmonic maps and \cite{Shatah} for wave maps to $\mathbb S^{L-1}$.}
to show that
\begin{equation}\label{convergenceEL}
 \int_0^T E_{s,p}' (u^h(\cdot, t),u^h(\cdot, t)\times \psi(\cdot, t))
 \xrightarrow{h \to 0}\int_0^T E_{s,p}'(u(\cdot, t),u\times\psi(\cdot, t)).
 \end{equation}
To show (\ref{convergenceEL}),  first note that
 \begin{eqnarray}\label{cancellationID}
&&\langle u^h(x,t)-u^h(y,t), u^h(x,t)\times\psi(x,t)-u^h(y,t)\times\psi(y,t)\rangle\nonumber\\
&&=\langle (u^h(x,t)-u^h(y,t))\times u^h(x,t), \psi(x,t)\rangle\nonumber\\
&&\quad-\langle (u^h(x,t)-u^h(y,t))\times u^h(y,t),\psi(y,t)\rangle\nonumber\\
&&=\langle (u^h(x,t)-u^h(y,t))\times u^h(x,t), \psi(x,t)\rangle\\
&&\quad-\langle (u^h(x,t)-u^h(y,t))\times u^h(x,t),\psi(y,t)\rangle\nonumber\\
&&=\langle (u^h(x,t)-u^h(y,t))\times u^h(x,t), \psi(x,t)-\psi(y,t)\rangle,\nonumber
\end{eqnarray}
where we have used the fact 
$$(u^h(x,t)-u^h(y,t))\times (u^h(x,t)-u^h(y,t))=0.$$
For $v:\Omega_T\to \mathbb R^L$, set
$${\mathcal C}^p[v](x,y,t)=\frac{|v(x,t)-v(y,t)|^{p-2}}{|x-y|^{n+sp}}, \ \ (x,y,t)\in\Omega_T^2.$$
Then we can  rewrite $\int_0^TE_{s,p}'(u^h(\cdot, t),u^h(\cdot, t)\times\psi(\cdot, t))$ as follows.
\[
 \begin{split}
 &\int_0^TE_{s,p}'\big(u^h(\cdot, t), u^h(\cdot, t)\times\psi(\cdot, t)\big)= \\
 % =&\int_{\Omega}\int_{\Omega} \frac{|u(x)-u(y)|^{p-2}(u^i(x)-u^i(y)) (\varphi^i(x) - u^i(x) u^j(x) \varphi^j(x) -  \varphi^i(y) + u^i(y) u^j(y) \varphi^j(y))}{|x-y|^{n+sp}}\ dx \ dy\\
%&\int_{\Omega\times\Omega} \frac{|u^h(x,t)-u^h(y,t)|^{p-2}(u^h(x,t)-u^h(y,t)) (u^h(x, t)\times\psi(x) -u^h(y,t)\times\psi(y))}{|x-y|^{n+sp}}
&\int_{\Omega^2_T} {\mathcal C}^p[u^h](x,y,t)\langle(u^h(x,t)-u^h(y,t))\times u^h(x, t), \psi(x,t) -\psi(y,t)\rangle.
\end{split}
\]
From Lemma~\ref{la:gagliardounderweakconv} and (\ref{cancellationID}), we see that
\[
\begin{split}
&\int_{\Omega^2_T} {\mathcal C}^p[u^h](x,y,t)\langle(u^h(x,t)-u^h(y,t))\times u^h(x, t), \psi(x,t) -\psi(y,t)\rangle\\
&\xrightarrow{h\rightarrow 0}
\int_{\Omega^2_T}{\mathcal C}^p[u](x,y,t)\langle(u(x,t)-u(y,t))\times u(x, t), \psi(x,t) -\psi(y,t)\rangle\\
&\qquad=\int_0^T E_{s,p}'(u(\cdot, t), u(\cdot, t)\times \psi(\cdot, t)).
\end{split}
\]
Thus we obtain that 
\begin{eqnarray} \label{eq:flow2}
&&\int_{\Omega_T} \langle\partial_t u(x,t), u(\cdot, t)\times \psi(\cdot, t)\rangle+\int_0^T E_{s,p}'(u(\cdot, t), u(\cdot, t)\times \psi(\cdot, t))\\
&&=0\nonumber 
\end{eqnarray}
holds for all $\psi\in L^\infty(0,T; W^{s,p}_0(\Omega, \mathbb R^L))\cap L^\infty(\Omega_T,\mathbb R^L)$.  

Finally, we claim
that (\ref{eq:flow2}) implies (\ref{eq:flow1}).  In fact, for any $\varphi\in L^\infty(0,T; W^{s,p}_0(\Omega, T_{u(\cdot, t)}\mathbb S^{L-1}))\cap L^\infty(\Omega_T,\mathbb R^L)$,
if we define 
$$\psi(\cdot, t)=-u(\cdot,t)\times \varphi(\cdot, t),$$ 
then we have 
$\psi\in L^\infty(0,T; W^{s,p}_0(\Omega, \mathbb R^{L}))\cap L^\infty(\Omega_T,\mathbb R^L)$, and
\[
\begin{split}
&u(\cdot, t)\times\psi(\cdot, t)=(u(\cdot,t)\times\varphi(\cdot, t))\times u(\cdot, t)\\
&=\langle u(\cdot, t), u(\cdot, t)\rangle\varphi(\cdot, t)-\langle u(\cdot, t),\varphi(\cdot, t)\rangle u(\cdot, t)\\
&=\varphi(\cdot, t),
\end{split}
\]
where we have used $\langle u(\cdot, t), u(\cdot, t)\rangle=0$, and $\langle u(\cdot, t),\varphi(\cdot, t)\rangle=0$ since $\varphi(\cdot, t)\in T_{u(\cdot, t)}\mathbb S^{L-1}$.
Thus (\ref{eq:flow1}) follows from (\ref{eq:flow2}).  This completes the proof when $\mathcal N=\mathbb S^{L-1}$.
We will give a proof for $\mathcal N$ a compact Riemannian homogeneous manifold in the next section.
\qed

\section{Proof of Theorem~\ref{main}: Step 2 - the case of homogeneous  Riemannian manifolds}\label{sec:homo}	
In this section, we will show that theorem 1.1 also holds when the target manifold $\mathcal N$ is a compact Riemannian
homogeneous manifold that is equivariantly embedded in $\mathbb R^L$. 

First we recall the following property on compact Riemannian homogeneous manifolds, which was 
proven by \cite{Moore-Schlafly1980}, see also \cite{freire}.
\begin{theorem}[Moore-Schlafly]
For any compact Riemannian manifold $\mathcal{N}$ with a compact Lie group $G$ acting on it by isometries we find for some $L \in \mathbb{N}$ an orthogonal representation $\rho$ of $G$ to the isometric group of an Euclidean space $\mathbb R^L$ and an isometric embedding $\Psi: \mathcal{N} \to \R^L$ which is equivariant with respect to $\rho$, that is $\Psi(g p) = \rho(g)\ \Psi(p)$
for all $p\in\mathcal N$ and $g\in G$.
\end{theorem}

With Theorem 5.1, we can follow the arguments by \cite{freire} pages 527-528 to show that,  by assuming $\mathcal N\subset \mathbb R^L$ and $G\subset {\rm{Iso}}(\mathbb R^L)$, the isometric group of $\mathbb R^L$, 
for any $G$-killing vector field $X$ on $\mathcal N$, there exists a killing vector field $\widehat{X}$ on
$\mathbb R^L$ with respect to ${\rm{Iso}}(\mathbb R^L)$ such that
\[
X=\widehat{X} \big|_{\mathcal N}.
\]
It follows from \cite{Helein} Lemma 2 that 
there exist a family of $G$-killing vector fields $\{X_\alpha\}_{\alpha=1}^l$ on $\mathcal N$, with $l={\rm{dim}}(G)$,  
and another family of
smooth vector fields $\{Y_\alpha\}_{\alpha=1}^l$ on $\mathcal N$ such that for any $y\in\mathcal N$, it holds that
\begin{equation}\label{decom}
 v = \sum_{\alpha = 1}^l\langle X_\alpha(y), v \rangle Y_\alpha(y) \quad \forall\ v \in T_y\mathcal{N}.
\end{equation}
We use crucially the following property of our Killing fields:
\begin{equation}\label{killthekill}
 \langle X_\alpha(p) - X_\alpha(q),p-q\rangle = 0 \quad \forall p,q \in \mathcal{N},\ \alpha = 1,\ldots,l.
\end{equation}
\begin{proof}[Proof of \eqref{killthekill}]In fact,  let $\{\widehat{X_\alpha}\}_{\alpha=1}^p$ be a family of killing vector fields on $\mathbb R^L$, with respect to
${\rm{Iso}}(\mathbb R^L)$, such that 
\[
X_\alpha(y)=\widehat{X_\alpha}(y), \ \forall \ y\in \mathcal N, \ 1\le\alpha\le l.
\]
Then we have
\[\begin{split}
\langle X_\alpha(p) - X_\alpha(q),p-q\rangle &= \langle \widehat{X_\alpha}(p) -\widehat{X_\alpha}(q),p-q\rangle \\
&=\langle D\widehat{X_\alpha}(p_*)(p-q), p-q\rangle\\
&=0,
\end{split}
\]
for some point $p_*$ in the line segment $[p,q]$, where we have used the fact that 
$D\widehat{X_\alpha}$ is skew-symmetric in the last step.
\end{proof}

%In particular, without loss of generality we may assume that $\mathcal{N} \subset \R^L$ has the form
%\[
% \mathcal{N} = \left \{ Pe: \quad P \in G \subset SO(L) \right \}
%\]
%for $G$ some closed subgroup of $SO(L)$ and some vector $e \in \R^L$.

%In particular, any element of the tangential space $T_{Pe} \mathcal{N}$ has the form $\omega P e$ for some $\omega \in %so(L)$.
%Let $l = \dim(\mathcal{N})$ and let $\omega_1, \ldots, \omega_l \in so(L)$ be so that $(\omega_\alpha e)_{\alpha=1}^l$ form %an orthonormal basis of $T_e \mathcal{N}$. The Killing fields $X(p) \in T_p \mathcal{N}$ are then defined as follows:
%\[
%X_\alpha(p) :=  \omega_\alpha p.
%\]
%Note that $(X_i(p))_{i=1}^l$ forms a basis of $T_p\mathcal{N}$, and clearly
%\begin{equation}\label{killthekill}
% \langle X_\alpha(p) - X_\alpha(q),p-q\rangle = 0 \quad \forall p,q \in \mathcal{N},\ \alpha = 1,\ldots,l.
%\end{equation}
%Also, we can recover any vector $v \in T_p \mathcal{N}$ from $\langle X_\alpha,v\rangle$. In particular, we have $(X_%\alpha)_{\alpha = 1}^l$ and $(Y_\alpha)_{\alpha = 1}^l$ tangent vector fields on $\mathcal{N} \subset \R^L$ so that 
%\[
% v = \sum_{\alpha = 1}^l\langle X_\alpha, v \rangle Y_\alpha \quad \forall v \in T_p\mathcal{N}.
%\]

We also need the following welll-known fact. See also \cite[Appendix 2.12.3]{SimonETH}.
\begin{lemma}\label{la:Projection}
Let $\Pi_\mathcal{N}:\mathcal{N}_\delta \to \mathcal{N} \subset \R^L$ be the orthogonal projection from the $\delta$- tubular neighbourhood of $\mathcal{N}$ into $\mathcal{N}$ for some small $\delta>0$. 
For $p\in\mathcal N$, set $P_{\mathcal N}(p) \in \R^{L \times L}$ by
\[
P_\mathcal{N}(p)(v) \equiv \frac{d}{dt} \Big|_{t = 0}\Pi_\mathcal{N}(p+tv), \ v\in\mathbb R^L.
\]
Then
\[
 P_\mathcal{N}(p)(v) = v \quad \forall v \in T_p \mathcal{N}, \quad P_\mathcal{N}(p)(v)= 0\quad \forall v \in (T_p \mathcal{N})^\perp. 
\]
In particular, $P_\mathcal{N}(p)^2 = P_\mathcal{N}(p)$ and $(P_\mathcal{N}(p))^T = P_\mathcal{N}(p)$ for any $p\in\mathcal N$.
\end{lemma}
\begin{proof} Indeed, fix $p \in \mathcal{N}$, and let us write $P$ instead of $P_\mathcal{N}(p)$. Let $o_1,\ldots,o_l$ be an orthonormal basis of $T_p \mathcal{N}$ and ${o}_{l+1},\ldots,{o}_{L}$ be an orthonormal basis of 
$(T_p\mathcal{N})^\perp$.
Then, since $P{o}_\beta = P{o}_\alpha = 0$ for $\alpha,\beta\in\{l+1 \ldots  N\}$,
\[
 \langle P{o}_\alpha,{o}_\beta \rangle = 0 = \langle {o}_\alpha,P{o}_\beta \rangle  
 \quad \forall \alpha,\beta = l+1,\ldots,N.
\]
Also, since $P o_\alpha = o_\alpha$ for $\alpha = 1, \ldots, l$,
\[
 \langle Po_\alpha,{o}_\beta \rangle = \langle o_\alpha,{o}_\beta \rangle = 0 = \langle o_\alpha,P{o}_\beta \rangle \quad \forall \alpha=1,\ldots l;\, \beta = l+1,\ldots,N.
\]
Finally,
\[
 \langle Po_\alpha,o_\beta \rangle = \langle o_\alpha,o_\beta \rangle = \langle o_\alpha,Po_\beta \rangle\quad \forall \alpha,\beta=1,\ldots l.
\]
Finally let $v, w \in \R^n$ then we have $v = \sum_{\alpha = 1}^N \lambda_\alpha o_\alpha$, $w = \sum_{\beta = 1}^N \mu_\beta o_\beta$.
Consequently,
\[
 \langle Pv,w \rangle = \sum_{\alpha,\beta = 1}^N \lambda_\alpha \mu_\beta \langle Po_i, o_j \rangle = \sum_{\alpha,\beta = 1}^N \lambda_\alpha \mu_\beta \langle o_\alpha, Po_\beta \rangle = \langle Pv,w \rangle.
\]
Thus, $P$ is symmetric.
\end{proof}

Note that then for any $v,w \in \R^N$, since $P_\mathcal{N}^T = P_\mathcal{N}$ and $P_\mathcal{N}X_\alpha = X_\alpha$, it follows from (\ref{decom})
that

\[
\begin{split}
 w^T P_\mathcal{N}v  &= \sum_{\alpha = 1}^l \langle X_\alpha, P_\mathcal{N}v \rangle \langle Y_\alpha, w \rangle= \sum_{\alpha = 1}^l \langle P_\mathcal{N}X_\alpha, v \rangle \langle Y_\alpha, w \rangle\\
 &= \sum_{\alpha = 1}^l \langle X_\alpha, v \rangle \langle Y_\alpha, w \rangle.
\end{split}
\]
Thus,
\begin{equation}\label{Pdecomp}
\left (P_\mathcal{N}(p) \right )_{ij} = \sum_{\alpha = 1}^l X_\alpha^i(p) Y_\alpha^j(p) = \sum_{\alpha = 1}^l Y_\alpha^i(p) X_\alpha^j(p), \forall \ p\in\mathcal N,
\end{equation}
where the last equality holds because $P(p)_{ij} = P(p)_{ji}$.

With these preparations, we are ready to show Step 2 in the proof of Theorem 1.1 for
a Riemannian homogeneous manifold $\mathcal N$ that is equivariantly  embedded into   $\mathbb R^L$.  

Let $v^h$, $u^h$, $u$ be as constructed in Step 1, Section~\ref{sec:step1}.

It suffices to show that for any $0<T<\infty$,
\begin{equation}\label{eq:flow3}
\int_{\Omega_T} \langle\partial_t u,P_{\mathcal{N}}(u)(\varphi)\rangle+\int_0^TE_{s,p}'\big(u(\cdot, t), P_\mathcal{N}(u(\cdot,t))(\varphi(\cdot, t))\big)=0, 
\end{equation}
for all $\varphi\in L^\infty(0,T; W^{s,p}_0(\Omega, \mathbb R^L))\cap L^\infty(\Omega_T,\mathbb R^L)$.

To simplify the presentation, we set 
$$
\mathcal W_h(x,y,t)=\frac{|u^h(x,t)-u^h(y,t)|^{p-2} (u^h(x,t)-u^h(y,t))}{|x-y|^{n+sp}},
$$
and
$$
\mathcal W(x,y,t)=\frac{|u(x,t)-u(y,t)|^{p-2} (u(x,t)-u(y,t))}{|x-y|^{n+sp}}.
$$
From (\ref{discreteEL1}), we see that 
\begin{eqnarray*}
&&-\int_{\Omega_T} \langle \partial_t v^h(x, t), X_\alpha(u^h(x, t))\eta(x,t)\rangle\\
&& =\int_{\Omega^2_T} \langle \mathcal W_h(x,y,t), X_\alpha(u^h(x,t))\eta(x,t)-X_\alpha(u^h(y,t))\eta(y,t)\rangle
\end{eqnarray*}
holds for any $1\le \alpha\le l$ and $\eta\in L^\infty(0,T; W^{s,p}_0(\Omega,\mathbb R))\cap L^\infty(\Omega_T,\mathbb R)$.

From (\ref{killthekill}), we have
$$\langle \mathcal W_h(x,y,t), X_\alpha(u^h(x,t))\eta(x,t)\rangle=\langle \mathcal W_h(x,y,t), X_\alpha(u^h(y,t))\eta(x,t)\rangle,$$
so that 
\[\begin{split}
&\int_{\Omega^2} \langle \mathcal W_h(x,y,t), X_\alpha(u^h(x,t))\eta(x,t)-X_\alpha(u^h(y,t))\eta(y,t)\rangle\\
&=\int_{\Omega^2} \langle \mathcal W_h(x,y,t), X_\alpha(u^h(y,t))(\eta(x,t)-\eta(y,t))\rangle\\
&\xrightarrow{h\rightarrow 0}\int_{\Omega^2} \langle \mathcal W(x,y,t), X_\alpha(u(y,t))(\eta(x,t)-\eta(y,t))\rangle\\
&=\int_{\Omega^2} \langle \mathcal W(x,y,t), X_\alpha(u(x,t))\eta(x,t)-X_\alpha(u(y,t))\eta(y,t)\rangle,
\end{split}
\]
where we have used Lemma \ref{la:gagliardounderweakconv} and (\ref{killthekill}) in the last two steps. 

It is straightforward that
\[
\begin{split}
&\int_{\Omega_T} \langle \partial_t v^h(x, t), X_\alpha(u^h(x, t))\eta(x,t)\rangle\\
&\xrightarrow{h\rightarrow 0}\int_{\Omega_T} \langle \partial_t u(x, t), X_\alpha(u(x, t))\eta(x,t)\rangle.
\end{split}
\]
Equalling these two limits yields that 
\begin{eqnarray}\label{eq:flow3a}
&&\int_{\Omega_T} \langle \partial_t u(x, t), X_\alpha(u(x, t))\eta(x,t)\rangle+\nonumber\\
&&\int_{\Omega^2_T} \langle \mathcal W(x,y,t), X_\alpha(u(x,t))\eta(x,t)-X_\alpha(u(y,t))\eta(y,t)\rangle=0,
\end{eqnarray} 
for all $1\le\alpha\le l$ and $\eta\in L^\infty(0,T; W^{s,p}_0(\Omega,\mathbb R))\cap L^\infty(\Omega_T,\mathbb R)$.

Now for any $\varphi\in L^\infty(0,T; W^{s,p}_0(\Omega,\mathbb R^L))\cap L^\infty(\Omega_T,\mathbb R^L)$ 
and $1\le \alpha\le l$
substitute $\eta=\eta_\alpha\equiv\langle Y_\alpha(u),\varphi\rangle$ into (\ref{eq:flow3a}) and take summation of
the resulting equations over $1\le\alpha\le l$. Observe that in view of \eqref{Pdecomp}, 
\[
 \sum_{\alpha=1}^l X_\alpha(u(x,t))\, \eta_\alpha(x,t)= P(u(x,t))\varphi(x,t).
\]
Thus, we arrive at
\[\begin{split}
&\int_{\Omega_T} \langle \partial_t u(x, t), P(u(x, t))\varphi(x,t)\rangle\\
=&\int_{\Omega^2_T} \langle \mathcal W(x,y,t), P(u(x, t))\varphi(x,t)-P(u(y, t))\varphi(y,t)\rangle=0,
\end{split}
\]
for any $\varphi\in L^\infty(0,T; W^{s,p}_0(\Omega,\mathbb R^L))\cap L^\infty(\Omega_T,\mathbb R^L)$. This shows \eqref{eq:flow3} and thus completes the proof of Theorem 1.1 for a Riemannian homogeneous manifold $\mathcal N$. \qed

\bigskip
\bigskip

\bibliographystyle{alpha}
\bibliography{biblio}

\end{document}